\documentclass[11pt,a4paper]{article}
\usepackage[T1]{fontenc}
\usepackage{lmodern}
\usepackage{amsmath,amssymb,amsthm}
\usepackage[a4paper,left=27mm,right=27mm,top=27mm,bottom=28mm,headheight=14pt]{geometry}
\usepackage{microtype}
\usepackage{fancyhdr}
\usepackage{url}
\usepackage[hidelinks]{hyperref}

\numberwithin{equation}{section}
\newtheorem{theorem}{Theorem}[section]
\newtheorem{lemma}[theorem]{Lemma}
\theoremstyle{definition}
\newtheorem{definition}[theorem]{Definition}
\allowdisplaybreaks[2]
\hypersetup{pdftitle={A Proof of the Common Root Conjecture for Legendre Polynomials},pdfauthor={Zikang Deng},pdfsubject={Stieltjes' common root conjecture}}

\title{\bfseries A Proof of the Common Root Conjecture\\for Legendre Polynomials}
\author{Zikang Deng\\[3pt]\normalsize Beijing Normal University}
\date{}

\begin{document}
\maketitle
\begin{abstract}
We prove Stieltjes' common root conjecture: Legendre polynomials of distinct degrees have no common nonzero root. We construct an auxiliary polynomial and show that, for the relevant power of two $q\ge2$, the edge of slope $1/q$ in its $2$-adic Newton polygon has horizontal length less than $3q$. If a common nonzero root existed, Newton polygon theory and the Legendre differential equation would force the same edge to have horizontal length at least $3q$, yielding a contradiction.
\end{abstract}

\section{Introduction}

The Legendre polynomials are defined by Rodrigues' formula
\begin{equation}\label{eq:1-1}
 P_n(x)=\frac{1}{2^n n!}\frac{d^n}{dx^n}(x^2-1)^n,
 \qquad n\ge0, 
\end{equation}
and satisfy $P_n(1)=1$. Stieltjes' common root conjecture asserts that Legendre polynomials of distinct degrees have no common root, apart from the root $0$ shared by two polynomials of odd degree.

In his letter no.~275 to Hermite, dated 2 October 1890, Stieltjes posed the common root problem and the irreducibility problem~\cite{Stieltjes,Gichev}. The latter asserts that $P_{2j}$ and $P_{2j+1}/x$ are irreducible over $\mathbb{Q}$ for every $j\ge1$, and hence implies the common root conjecture. Wahab~\cite{Wahab} determined the $2$-adic Newton polygons of the shifted Legendre polynomials and proved irreducibility in several cases; Cullinan and Hajir~\cite{CullinanHajir} further studied their Galois groups. Concerning common roots, in a talk on 18 February 2026, Mangoubi presented a result obtained jointly with Kadets and Weller Weiser: every fixed nonzero number is a root of at most finitely many Legendre polynomials~\cite{Mangoubi}.

The common root problem is also related to the nodal sets of spherical eigenfunctions. In their paper in \emph{Inventiones Mathematicae}, Bourgain and Rudnick~\cite[Section~1.4]{BourgainRudnick} discussed whether a non-equatorial circle of latitude can be contained in the nodal sets of spherical harmonics corresponding to infinitely many distinct eigenvalues, and pointed out the direct connection between the axisymmetric case and Stieltjes' conjecture. On the unit sphere $S^2$, with colatitude $\theta$ and longitude $\varphi$, axisymmetric spherical harmonics about a fixed polar axis may be chosen as
\begin{equation}\label{eq:1-2}
 Z_j(\theta,\varphi)=P_j(\cos\theta), 
\end{equation}
with eigenvalue $j(j+1)$ for the nonnegative Laplacian. For $0<\theta_0<\pi$, the circle of latitude $\theta=\theta_0$ is contained in the nodal set of $Z_j$ if and only if $P_j(\cos\theta_0)=0$. Thus, two axisymmetric spherical harmonics of distinct degrees vanish identically on the same non-equatorial circle of latitude precisely when the corresponding Legendre polynomials have a common nonzero real root. Our result shows that, once the polar axis is fixed, each non-equatorial circle of latitude corresponds to at most one such degree.

Our main result is the following.

\begin{theorem}[Stieltjes' common root conjecture]\label{thm:1-1}
Let $0\le m<n$. Then $P_m$ and $P_n$ have no common nonzero complex root. With the greatest common divisor in $\mathbb{Q}[x]$ normalized to be monic, we have
\[
 \gcd(P_m,P_n)=
 \begin{cases}
 x,&\text{if both $m$ and $n$ are odd},\\
 1,&\text{otherwise}.
 \end{cases}
\]
\end{theorem}

To prove Theorem~\ref{thm:1-1}, define the shifted polynomials and the auxiliary polynomial by
\begin{align}
 Q_j(X)&=P_j(2X+1),\label{eq:1-3}\\
 I_{m,n}(X)&=\frac1X\int_0^X Q_m(t)Q_n(t)\,dt.\label{eq:1-4}
\end{align}
Here the integral is taken term by term. Since the resulting polynomial has zero constant term, $I_{m,n}\in\mathbb{Q}[X]$; moreover, $Q_j(0)=1$ gives $I_{m,n}(0)=1$.

The key step is to prove that, if $q=2^e\ge2$ and the $e$th binary digits of both $m$ and $n$ are $1$, then the edge of slope $1/q$ in the $2$-adic Newton polygon of $I_{m,n}$ has horizontal length less than $3q$. Section~3 establishes this estimate in four steps. We first use a known coefficient point to obtain an upper bound for the intercept of a lower supporting line of the Newton polygon, and then derive a lower bound. Combining the two bounds yields conditions on the abscissas of the endpoints of the edge of slope $1/q$; these conditions then bound its horizontal length.

Section~4 deduces the main theorem. If a common nonzero root existed, then $Q_m$ and $Q_n$ would have a common irreducible factor of degree at least $q$ in $\mathbb{Q}_2[X]$. The Legendre differential equation ensures that the cube of this factor divides $I_{m,n}$, forcing the corresponding edge to have horizontal length at least $3q$, contrary to the estimate above.

\section{Preliminaries}

\subsection{Valuations, the floor function, and binary identities}

For a nonzero rational number $a=2^g u/v$, where $u$ and $v$ are odd,
define $\nu_2(a)=g$, and set $\nu_2(0)=+\infty$. We write $\nu=\nu_2$
throughout. Let $\mathbb{Q}_2$ denote the field of $2$-adic numbers, and
use the same notation $\nu$ for the extension of the valuation to its
algebraic closure.

For a real number $x$, square brackets $[x]$ denote the greatest integer
not exceeding $x$, that is, the floor function:
\begin{equation}
 [x]=\max\{z\in\mathbb{Z}:z\le x\}.
 \label{eq:2-1}
\end{equation}
For a nonnegative integer $j$, let $s(j)$ be the sum of its binary digits
and let $f(j)=\nu(j!)$, with $s(0)=f(0)=0$. Binary digits are indexed
starting with the units digit as digit $0$. We use the following
standard identities:
\begin{align}
 \nu(ab)&=\nu(a)+\nu(b),
 &\nu(a+b)&\ge\min\{\nu(a),\nu(b)\};
 \label{eq:2-2}\\
 [x+y]&\ge[x]+[y],
 &[x+d]&=[x]+d\quad(d\in\mathbb{Z});
 \label{eq:2-3}\\
 f(j)&=\sum_{i\ge1}\left[\frac{j}{2^i}\right]=j-s(j);
 &&\label{eq:2-4}\\
 f(2j)&=f(2j+1)=j+f(j),
 &\nu\binom{2j}{j}&=s(j);
 \label{eq:2-5}\\
 s(j+1)&=s(j)+1-\nu(j+1).
 &&\label{eq:2-6}
\end{align}
If $q=2^e$ and $A=qT+\eta$, where $T\ge0$ and $0\le\eta<q$ are integers,
then
\begin{equation}
 s(A)=s(T)+s(\eta),\qquad
 \sum_{j>e}\left[\frac{A}{2^j}\right]=f(T).
 \label{eq:2-7}
\end{equation}
For integers $A\ge B\ge0$, we also have
\begin{equation}
 \nu\binom{A}{B}
 =\sum_{j\ge1}\left\{
 \left[\frac{A}{2^j}\right]
 -\left[\frac{B}{2^j}\right]
 -\left[\frac{A-B}{2^j}\right]\right\},
 \label{eq:2-8}
\end{equation}
where each difference of floor functions is a nonnegative integer.

\subsection{Newton polygons}

\begin{definition}\label{def:2-1}
Let $F(X)=\sum_{k=0}^d c_kX^k\in\mathbb{Q}_2[X]$, with $c_0c_d\ne0$.
The lower boundary of the convex hull of the coefficient points
\[
 (k,\nu(c_k)),\qquad c_k\ne0,
\]
is called the $2$-adic Newton polygon of $F$, denoted by
$\operatorname{NP}_2(F)$. Throughout, adjacent collinear segments are
regarded as a single edge, and the length of an edge means its horizontal
length. Let $L_F(\lambda)$ denote the length of the edge of slope
$\lambda$, with $L_F(\lambda)=0$ if there is no such edge.
\end{definition}

\begin{theorem}[Root valuation theorem]\label{thm:2-2}
If $\operatorname{NP}_2(F)$ has an edge of slope $\lambda$ and horizontal
length $L$, then $F$ has exactly $L$ roots in $\overline{\mathbb{Q}}_2$
of valuation $-\lambda$, counted with multiplicity.
\end{theorem}

See \cite[Proposition~7.44, pp.~125--126]{Milne}. The coefficients there
are ordered by descending powers; with the ascending-power convention
used here, the root valuation is the negative of the edge slope.

\begin{theorem}[Product theorem]\label{thm:2-3}
Let $F,G\in\mathbb{Q}_2[X]$, with $F(0)G(0)\ne0$. The edges of
$\operatorname{NP}_2(FG)$ are obtained by merging the edges of the two
factors in increasing order of slope, adding the horizontal lengths
of edges having the same slope. In particular,
\[
 L_{FG}(\lambda)=L_F(\lambda)+L_G(\lambda).
\]
\end{theorem}

\begin{proof}
By Theorem~\ref{thm:2-2}, the length of the edge of slope $\lambda$
equals the total multiplicity of the roots of valuation $-\lambda$.
The multiplicity of each root in the product is the sum of its
multiplicities in the two factors, which proves the formula for the
edge lengths. The starting point is determined by
\[
 \nu\bigl(F(0)G(0)\bigr)=\nu(F(0))+\nu(G(0)),
\]
giving the stated description of the merging of edges in order of slope.
\end{proof}

A valuation on a complete discretely valued field extends uniquely to
each finite extension \cite[Theorem~7.38 and Corollary~7.40,
pp.~123--124]{Milne}. Thus all conjugate roots of an irreducible
polynomial over $\mathbb{Q}_2$ with nonzero constant term have the same
valuation. By Theorem~\ref{thm:2-2}, its Newton polygon has only one edge.

Fix a slope $\lambda$. By Definition~\ref{def:2-1}, the lower supporting
line in this direction is
\begin{equation}
 y=\lambda x+\min_{c_k\ne0}\{\nu(c_k)-\lambda k\}.
 \label{eq:2-9}
\end{equation}
If this direction corresponds to an edge, both endpoints attain the
minimum in this expression.

By Rodrigues' formula, or by setting the Jacobi parameters
$\alpha=\beta=0$ and substituting $x=2X+1$ in
\cite[Eq.~(18.5.7)]{DLMF}, we obtain
\begin{equation}
 Q_n(X)=\sum_{k=0}^n\binom{n}{k}\binom{n+k}{k}X^k.
 \label{eq:2-10}
\end{equation}
In particular, $Q_n\in\mathbb{Z}[X]$ and $Q_n(0)=1$. Its Newton polygon
is determined by the following theorem.

\begin{theorem}[Wahab]\label{thm:2-4}
Let $n\ge1$, and write $n=q_1+\cdots+q_s$, where
$q_1>\cdots>q_s$ are powers of $2$. From left to right, the edges of
$\operatorname{NP}_2(Q_n)$ have slopes $1/q_i$, horizontal lengths
$q_i$, and vertical increments $1$.
\end{theorem}

This is \cite[Theorem~3.1]{Wahab}; see also its restatement in
\cite[Theorem~7.3]{CullinanHajir}.

\subsection{Gaunt's formula and the coefficients of the integral}

We use Gaunt's formula to expand $Q_mQ_n$ as a linear combination of
shifted Legendre polynomials, then integrate term by term to obtain
the coefficients of $I_{m,n}$.

Fix $0\le m<n$. Gaunt's formula
\cite[Eqs.~(34.3.5) and (34.3.19)]{DLMF} gives
\begin{equation}
 Q_m(X)Q_n(X)=\sum_{r=0}^m\gamma_rQ_{\ell_r}(X),
 \qquad\ell_r=m+n-2r,
 \label{eq:2-11}
\end{equation}
where
\begin{equation}
 \gamma_r=
 \frac{2\ell_r+1}{2m+2n-2r+1}
 \frac{\displaystyle
 \binom{2m-2r}{m-r}\binom{2n-2r}{n-r}\binom{2r}{r}}
 {\displaystyle\binom{2m+2n-2r}{m+n-r}}.
 \label{eq:2-12}
\end{equation}
To express the coefficients of $I_{m,n}$, define
\begin{equation}
 B(\ell,k)=\frac{1}{k+1}\binom{\ell}{k}\binom{\ell+k}{k}
 \quad(0\le k\le\ell),\qquad
 \operatorname{Cat}_k=\frac{1}{k+1}\binom{2k}{k},
 \label{eq:2-13}
\end{equation}
where $\operatorname{Cat}_k$ is the $k$th Catalan number. Write
\begin{equation}
 I_{m,n}(X)=\sum_{k=0}^{m+n}a_kX^k.
 \label{eq:2-14}
\end{equation}

\begin{lemma}\label{lem:2-5}
The above coefficients satisfy
\begin{equation}
 a_k=\sum_{\substack{0\le r\le m\\k\le\ell_r}}
 \gamma_rB(\ell_r,k).
 \label{eq:2-15}
\end{equation}
Moreover,
\begin{align}
 \nu(\gamma_r)
 &=f(m+n-r)-f(m-r)-f(n-r)-f(r),
 \label{eq:2-16}\\
 B(\ell,k)&=\operatorname{Cat}_k\binom{\ell+k}{2k},
 \label{eq:2-17}\\
 \nu(\operatorname{Cat}_k)&=s(k+1)-1.
 \label{eq:2-18}
\end{align}
\end{lemma}

\begin{proof}
We first compute the coefficients of $I_{m,n}$. Substitute
\eqref{eq:2-10} into \eqref{eq:2-11} and integrate as in
\eqref{eq:1-4}:
\begin{align*}
 I_{m,n}(X)
 &=\sum_{r=0}^m\gamma_r\frac{1}{X}\int_0^XQ_{\ell_r}(t)\,dt\\
 &=\sum_{r=0}^m\gamma_r\sum_{k=0}^{\ell_r}
 \frac{1}{k+1}\binom{\ell_r}{k}\binom{\ell_r+k}{k}X^k\\
 &=\sum_{k=0}^{m+n}\left(
 \sum_{\substack{0\le r\le m\\k\le\ell_r}}
 \gamma_rB(\ell_r,k)\right)X^k.
\end{align*}
Comparing coefficients gives \eqref{eq:2-15}.

Next, we compute the valuation of $\gamma_r$. The numerator and
denominator of the first fraction in \eqref{eq:2-12} are both odd.
By \eqref{eq:2-5} and \eqref{eq:2-4},
\begin{align*}
 \nu(\gamma_r)
 &=s(m-r)+s(n-r)+s(r)-s(m+n-r)\\
 &=(m-r-f(m-r))+(n-r-f(n-r))\\
 &\qquad+(r-f(r))-(m+n-r-f(m+n-r))\\
 &=f(m+n-r)-f(m-r)-f(n-r)-f(r).
\end{align*}
Finally, we factor $B(\ell,k)$ as the product of a Catalan number and
a binomial coefficient. By \eqref{eq:2-13},
\begin{align*}
 B(\ell,k)
 &=\frac{(\ell+k)!}{(k+1)(k!)^2(\ell-k)!}\\
 &=\frac{(2k)!}{(k+1)(k!)^2}
 \frac{(\ell+k)!}{(2k)!(\ell-k)!}
 =\operatorname{Cat}_k\binom{\ell+k}{2k}.
\end{align*}
Using \eqref{eq:2-5} and \eqref{eq:2-6} gives
\begin{align*}
 \nu(\operatorname{Cat}_k)
 &=f(2k)-2f(k)-\nu(k+1)\\
 &=k-f(k)-\nu(k+1)\\
 &=s(k)-\nu(k+1)=s(k+1)-1.
\end{align*}
\end{proof}

\section{The edge-length estimate}

\begin{theorem}\label{thm:3-1}
Let $0\le m<n$ and $q=2^e\ge2$, and suppose that the $e$th binary digits of both $m$ and $n$ are $1$. Then
\begin{equation}
L_{I_{m,n}}(1/q)<3q.
\label{eq:3-1}
\end{equation}
\end{theorem}
Since this edge length is an integer multiple of $q$, we also have $L_{I_{m,n}}(1/q)\le2q$.

Throughout this section, fix $m,n,q$ satisfying the hypotheses of the theorem, and write
\begin{equation}
m=qM+u,\qquad n=qN+v,\qquad 0\le u,v<q,\qquad h_0=-f(M)-f(N).
\label{eq:3-2}
\end{equation}
Thus $M$ and $N$ are positive odd integers. The summation indices $r,k$ below are integers satisfying
\begin{equation}
0\le r\le m,\qquad 0\le k\le\ell_r=m+n-2r.
\label{eq:3-3}
\end{equation}
The proof has four steps. We first use a known coefficient point to obtain an upper bound for the intercept of a lower supporting line of the Newton polygon, and then establish a lower bound. Combining the two bounds gives conditions on the horizontal coordinates of the endpoints of the edge of slope $1/q$. These conditions then bound its horizontal length.

\subsection{The upper-bound inequality}

\begin{lemma}\label{lem:3-2}
We have
\begin{equation}
\nu\bigl(a_{q(M+N)}\bigr)-\frac{q(M+N)}{q}=h_0.
\label{eq:3-4}
\end{equation}
If $k$ is the degree of an endpoint of the edge of slope $1/q$ in $\operatorname{NP}_2(I_{m,n})$, then there is an $r$ such that
\begin{equation}
\nu\bigl(\gamma_r B(\ell_r,k)\bigr)-\frac{k}{q}\le h_0.
\label{eq:3-5}
\end{equation}
\end{lemma}

\begin{proof}
To determine the coefficient in \eqref{eq:3-4}, we first find the corresponding vertex of the Newton polygon of $Q_mQ_n$. By Theorem~\ref{thm:2-4}, a $1$ in the $j$th binary position corresponds to an edge of length $2^j$ and slope $2^{-j}$, with vertical increment $1$. Write
\[
m=\sum_{j\ge0}\varepsilon_j2^j,\qquad\varepsilon_j\in\{0,1\}.
\]
The edges of slope at most $1/q$ satisfy
\[
2^{-j}\le2^{-e}\iff j\ge e,\qquad
\sum_{j\ge e}\varepsilon_j2^j=qM,\qquad
\sum_{j\ge e}\varepsilon_j=s(M).
\]
Since $Q_m(0)=1$, the Newton polygon starts at $(0,0)$, so the right endpoint of its edge of slope $1/q$ is $(qM,s(M))$. Similarly, the corresponding right endpoint for $Q_n$ is $(qN,s(N))$. By the multiplication theorem, Theorem~\ref{thm:2-3}, the corresponding right endpoint for $Q_mQ_n$ is
\[
\bigl(q(M+N),s(M)+s(N)\bigr).
\]
This is an actual coefficient point. By \eqref{eq:1-4} and \eqref{eq:2-14},
\[
Q_mQ_n=(XI_{m,n})'=\sum_{k=0}^{m+n}(k+1)a_kX^k,
\]
and hence
\[
\nu\bigl((q(M+N)+1)a_{q(M+N)}\bigr)=s(M)+s(N).
\]
Since $q(M+N)+1$ is odd, it follows that
\[
\begin{aligned}
\nu\bigl(a_{q(M+N)}\bigr)-\frac{q(M+N)}q
&=s(M)+s(N)-M-N\\
&=-f(M)-f(N)=h_0.
\end{aligned}
\]

Now let $k$ be the degree of an endpoint of the edge. Denote the intercept of the lower supporting line of slope $1/q$ by
\[
h=\min_{a_k\ne0}\left\{\nu(a_k)-\frac{k}{q}\right\}.
\]
By \eqref{eq:3-4}, $h\le h_0$. Substituting the coefficient expansion \eqref{eq:2-15} into the valuation inequality \eqref{eq:2-2}, and using the endpoint relation \eqref{eq:2-9}, gives
\[
\min_{\substack{0\le r\le m\\k\le\ell_r}}
\left\{\nu\bigl(\gamma_rB(\ell_r,k)\bigr)-\frac{k}{q}\right\}
\le\nu(a_k)-\frac{k}{q}=h\le h_0.
\]
One of the terms attains the minimum, proving \eqref{eq:3-5}.
\end{proof}

\subsection{The lower-bound inequality}

To estimate the left-hand side of \eqref{eq:3-5}, we treat $\gamma_r$ and $B(\ell_r,k)$ separately. Write
\begin{equation}
r=qt+\rho,\qquad k+1=qK+w,\qquad0\le\rho,w<q,
\label{eq:3-6}
\end{equation}
and define
\begin{equation}
C_r=\sum_{j=1}^e\left\{
\left[\frac{m+n-r}{2^j}\right]
-\left[\frac{m-r}{2^j}\right]
-\left[\frac{n-r}{2^j}\right]
-\left[\frac{r}{2^j}\right]\right\}.
\label{eq:3-7}
\end{equation}

\begin{lemma}\label{lem:3-3}
For every pair $(r,k)$ satisfying \eqref{eq:3-3},
\begin{equation}
\begin{aligned}
\nu\bigl(\gamma_rB(\ell_r,k)\bigr)-\frac{k}{q}
\ge{}&h_0-1+\frac1q+C_r+f(t)\\
&+\sum_{j>e}\left\{
\left[\frac{m+n-r}{2^j}\right]-\left[\frac{k+1}{2^j}\right]
\right\}\\
&+\nu\binom{\ell_r+k}{2k}.
\end{aligned}
\label{eq:3-8}
\end{equation}
On the right-hand side, $C_r$, $f(t)$, each difference of floor functions, and the valuation of the binomial coefficient are nonnegative integers.
\end{lemma}

\begin{proof}
We first estimate $\nu(\gamma_r)$. Substituting the factorial valuation formula \eqref{eq:2-4} into \eqref{eq:2-16}, and splitting the sum at $j=e$, gives
\begin{equation}
\nu(\gamma_r)=C_r+\sum_{j>e}\left\{
\left[\frac{m+n-r}{2^j}\right]
-\left[\frac{m-r}{2^j}\right]
-\left[\frac{n-r}{2^j}\right]
-\left[\frac r{2^j}\right]\right\}.
\label{eq:3-9}
\end{equation}
Since $m=(m-r)+r$ and $n=(n-r)+r$, \eqref{eq:2-3} yields
\[
\begin{aligned}
\left[\frac m{2^j}\right]&\ge\left[\frac{m-r}{2^j}\right]+\left[\frac r{2^j}\right],\\
\left[\frac n{2^j}\right]&\ge\left[\frac{n-r}{2^j}\right]+\left[\frac r{2^j}\right].
\end{aligned}
\]
Rearranging, we obtain
\begin{equation}
\begin{aligned}
-\left[\frac{m-r}{2^j}\right]&\ge-\left[\frac m{2^j}\right]+\left[\frac r{2^j}\right],\\
-\left[\frac{n-r}{2^j}\right]&\ge-\left[\frac n{2^j}\right]+\left[\frac r{2^j}\right].
\end{aligned}
\label{eq:3-10}
\end{equation}
Substitute \eqref{eq:3-10} into \eqref{eq:3-9}, and evaluate the sums using \eqref{eq:2-7}:
\begin{equation}
\begin{aligned}
\nu(\gamma_r)
&\ge C_r+\sum_{j>e}\left\{
\left[\frac{m+n-r}{2^j}\right]-\left[\frac m{2^j}\right]
-\left[\frac n{2^j}\right]+\left[\frac r{2^j}\right]\right\}\\
&=C_r-f(M)-f(N)+f(t)+\sum_{j>e}\left[\frac{m+n-r}{2^j}\right]\\
&=h_0+C_r+f(t)+\sum_{j>e}\left[\frac{m+n-r}{2^j}\right].
\end{aligned}
\label{eq:3-11}
\end{equation}

We next estimate $\nu(B(\ell_r,k))-k/q$. By \eqref{eq:2-17} and \eqref{eq:2-18}, substituting $k=qK+w-1$ and using \eqref{eq:2-7} together with $s(K)-K=-f(K)$, we obtain
\[
\begin{aligned}
\nu(B(\ell_r,k))-\frac{k}{q}
&=s(k+1)-1-\frac{k}{q}+\nu\binom{\ell_r+k}{2k}\\
&=s(K)+s(w)-1-K-\frac{w}{q}+\frac1q+\nu\binom{\ell_r+k}{2k}\\
&=-f(K)-1+\frac1q+s(w)-\frac{w}{q}+\nu\binom{\ell_r+k}{2k}.
\end{aligned}
\]
Here
\[
s(w)-\frac wq=0\quad(w=0),\qquad
s(w)-\frac wq\ge1-\frac wq>0\quad(1\le w<q).
\]
Dropping this nonnegative term gives
\begin{equation}
\nu(B(\ell_r,k))-\frac kq
\ge-f(K)-1+\frac1q+\nu\binom{\ell_r+k}{2k}.
\label{eq:3-12}
\end{equation}

Add \eqref{eq:3-11} and \eqref{eq:3-12}. By \eqref{eq:2-7},
\[
f(K)=\sum_{j>e}\left[\frac{k+1}{2^j}\right],
\]
so $-f(K)$ combines with the sum in \eqref{eq:3-11} to give the differences of floor functions in \eqref{eq:3-8}. This proves the required lower bound.

It remains to verify the nonnegativity of the additional terms in this bound. Since $m+n-r=(m-r)+(n-r)+r$, \eqref{eq:2-3} gives
\[
\left[\frac{m+n-r}{2^j}\right]
\ge\left[\frac{m-r}{2^j}\right]+\left[\frac{n-r}{2^j}\right]+\left[\frac r{2^j}\right],
\]
and hence $C_r\ge0$. Since $t!$ and the binomial coefficient are positive integers,
\[
f(t)\ge0,\qquad\nu\binom{\ell_r+k}{2k}\ge0.
\]
Thus it only remains to prove, for $j>e$, that
\begin{equation}
\left[\frac{m+n-r}{2^j}\right]\ge\left[\frac{k+1}{2^j}\right].
\label{eq:3-13}
\end{equation}
If $r\ge1$, then
\[
k+1\le\ell_r+1=m+n-2r+1\le m+n-r,
\]
and \eqref{eq:3-13} follows from the monotonicity of the floor function.

If $r=0$, then $k+1\le m+n+1$. Since $M+N$ is even and $1\le u+v+1<2q$,
\[
m+n+1=q(M+N)+(u+v+1)\not\equiv0\pmod{2q}.
\]
For $j>e$, the divisibility $2q\mid2^j$ implies $2^j\nmid m+n+1$, and hence
\[
\left[\frac{m+n}{2^j}\right]
=\left[\frac{m+n+1}{2^j}\right]
\ge\left[\frac{k+1}{2^j}\right].
\]
Thus \eqref{eq:3-13} also holds when $r=0$.
\end{proof}

\subsection{Conditions obtained by combining the bounds}

Combining \eqref{eq:3-5} and \eqref{eq:3-8} gives a sum of nonnegative integers that is strictly less than $1$. Consequently, all these terms must vanish.

\begin{lemma}\label{lem:3-4}
Suppose that the integer pair $(r,k)$ lies in the range \eqref{eq:3-3} and satisfies \eqref{eq:3-5}. Then
\begin{align}
C_r&=0,\qquad t\in\{0,1\},
\label{eq:3-14}\\
\left[\frac{m+n-r}{2^j}\right]&=\left[\frac{k+1}{2^j}\right]\qquad(j>e),
\label{eq:3-15}\\
\nu\binom{\ell_r+k}{2k}&=0.
\label{eq:3-16}
\end{align}
In particular,
\begin{equation}
\left[\frac{m+n-r}{2q}\right]=\left[\frac{k+1}{2q}\right].
\label{eq:3-17}
\end{equation}
\end{lemma}

\begin{proof}
Combining \eqref{eq:3-5} and \eqref{eq:3-8}, cancelling $h_0$, and rearranging gives
\begin{equation}
\begin{aligned}
C_r+f(t)+\sum_{j>e}\left\{
\left[\frac{m+n-r}{2^j}\right]-\left[\frac{k+1}{2^j}\right]\right\}\\
+\nu\binom{\ell_r+k}{2k}\le1-\frac1q<1.
\end{aligned}
\label{eq:3-18}
\end{equation}
By Lemma~\ref{lem:3-3}, all the terms on the left are nonnegative integers. They therefore all vanish; in particular, $C_r=f(t)=0$. Since
\[
f(0)=f(1)=0,\qquad f(t)=\nu(t!)\ge1\quad(t\ge2),
\]
we have $t\in\{0,1\}$. Taking $j=e+1$ in \eqref{eq:3-15}, and using $2^{e+1}=2q$, yields \eqref{eq:3-17}.
\end{proof}

\subsection{The span of the degrees}

\begin{lemma}\label{lem:3-5}
Suppose that the integer pairs $(r,k)$ and $(r',k')$ both lie in the range \eqref{eq:3-3} and satisfy \eqref{eq:3-5}. Then $r$ and $r'$ have the same remainder $\rho$ upon division by $q$. This remainder is uniquely determined by $m,n,q$, and
\begin{equation}
r\in\{\rho,\rho+q\},\qquad |r-r'|\le q.
\label{eq:3-19}
\end{equation}
Each such pair also satisfies
\begin{equation}
\ell_r-q<k\le\ell_r,\qquad0\le\ell_r-k\le q-1.
\label{eq:3-20}
\end{equation}
Consequently,
\begin{equation}
|k-k'|\le3q-1<3q.
\label{eq:3-21}
\end{equation}
\end{lemma}

\begin{proof}
We first prove that $\rho$ is uniquely determined by $m,n,q$. Let $a,b$ be the remainders of $m-r,n-r$, respectively, upon division by $q$. For $1\le j\le e$, let $a_j,b_j,\rho_j$ be the remainders of $m-r,n-r,r$, respectively, upon division by $2^j$. Then
\[
\begin{aligned}
m-r&=2^j\left[\frac{m-r}{2^j}\right]+a_j,\\
n-r&=2^j\left[\frac{n-r}{2^j}\right]+b_j,\\
r&=2^j\left[\frac r{2^j}\right]+\rho_j.
\end{aligned}
\]
Adding these identities, dividing by $2^j$, and taking floors gives
\[
\begin{aligned}
\left[\frac{m+n-r}{2^j}\right]
={}&\left[\frac{m-r}{2^j}\right]
+\left[\frac{n-r}{2^j}\right]
+\left[\frac r{2^j}\right]\\
&+\left[\frac{a_j+b_j+\rho_j}{2^j}\right].
\end{aligned}
\]
Substituting into \eqref{eq:3-7} and using \eqref{eq:3-14}, we obtain
\begin{equation}
0=C_r=\sum_{j=1}^e\left[\frac{a_j+b_j+\rho_j}{2^j}\right]
\quad\Longrightarrow\quad
a_j+b_j+\rho_j<2^j\quad(1\le j\le e).
\label{eq:3-22}
\end{equation}

Since $2^j\mid q$, these remainders retain the lowest $j$ bits of $a,b,\rho$, respectively. If two of these numbers, say $a,b$, both had a $1$ in the $i$th binary position, where $0\le i<e$, then taking $j=i+1$ would give
\[
a_{i+1}\ge2^i,\qquad b_{i+1}\ge2^i,\qquad
2^{i+1}\le a_{i+1}+b_{i+1}+\rho_{i+1}<2^{i+1},
\]
a contradiction. Thus, in each binary position, at most one of $a,b,\rho$ has a $1$. Taking $j=e$ in \eqref{eq:3-22}, we obtain
\begin{equation}
a+b+\rho<q.
\label{eq:3-23}
\end{equation}
Since $m=(m-r)+r$ and $n=(n-r)+r$,
\[
\begin{aligned}
u&\equiv a+\rho\pmod q,\qquad0\le u,a+\rho<q
\quad\Longrightarrow\quad u=a+\rho,\\
v&\equiv b+\rho\pmod q,\qquad0\le v,b+\rho<q
\quad\Longrightarrow\quad v=b+\rho.
\end{aligned}
\]
That is,
\begin{equation}
u=a+\rho,\qquad v=b+\rho.
\label{eq:3-24}
\end{equation}
Both additions involve no carries, so every position in which $\rho$ has a $1$ also has a $1$ in both $u$ and $v$. Conversely, if $u$ and $v$ both had a $1$ in some position where $\rho$ had a $0$, then \eqref{eq:3-24} would imply that $a$ and $b$ both had a $1$ in that position, a contradiction. Therefore, $\rho$ has a $1$ in a binary position if and only if both $u$ and $v$ have a $1$ there. This determines $\rho$ uniquely. Moreover, \eqref{eq:3-14} gives
\[
r=qt+\rho,\qquad t\in\{0,1\}
\quad\Longrightarrow\quad r\in\{\rho,\rho+q\},\qquad |r-r'|\le q.
\]

We next prove that $k>\ell_r-q$. Equation~\eqref{eq:3-17} first gives a lower bound; we then use \eqref{eq:3-16} to rule out its boundary case. Set
\[
H=M+N-2t.
\]
Since $M,N$ are positive odd integers and $t\in\{0,1\}$, we have $H\in2\mathbb Z_{\ge0}$. Substituting \eqref{eq:3-24}, we obtain
\begin{align}
\ell_r&=q(M+N-2t)+u+v-2\rho\notag\\
&=qH+(a+\rho)+(b+\rho)-2\rho=qH+a+b,
\label{eq:3-25}\\
m+n-r&=q(M+N-t)+u+v-\rho\notag\\
&=qH+qt+(a+\rho)+(b+\rho)-\rho=qH+qt+a+b+\rho.
\label{eq:3-26}
\end{align}
By \eqref{eq:3-23},
\[
0\le qt+a+b+\rho<2q,
\]
so
\[
\begin{aligned}
\left[\frac{m+n-r}{2q}\right]
&=\left[\frac H2+\frac{qt+a+b+\rho}{2q}\right]\\
&=\frac H2+\left[\frac{qt+a+b+\rho}{2q}\right]=\frac H2.
\end{aligned}
\]
Substituting into \eqref{eq:3-17} gives
\begin{equation}
\left[\frac{k+1}{2q}\right]=\frac H2
\quad\Longrightarrow\quad k+1\ge qH.
\label{eq:3-27}
\end{equation}
Suppose that $k\le\ell_r-q$. Using \eqref{eq:3-25} and $a+b\le q-1$, we get
\[
qH\le k+1\le\ell_r-q+1=qH+a+b-q+1\le qH.
\]
Equality holds throughout, so
\begin{equation}
k=qH-1,\qquad a+b=q-1,\qquad\ell_r=qH+q-1.
\label{eq:3-28}
\end{equation}
Since $k+1=qH\ge1$, we have $H>0$. By \eqref{eq:3-28},
\begin{equation}
\binom{\ell_r+k}{2k}
=\binom{2qH+q-2}{2qH-2}
=\binom{2qH+q-2}{q}.
\label{eq:3-29}
\end{equation}
All the terms in the binomial valuation formula \eqref{eq:2-8} are nonnegative. Retaining the term with denominator $2q=2^{e+1}$ gives
\[
\begin{aligned}
\nu\binom{2qH+q-2}{q}
&\ge\left[\frac{2qH+q-2}{2q}\right]
-\left[\frac q{2q}\right]-\left[\frac{2qH-2}{2q}\right]\\
&=\left[H+\frac{q-2}{2q}\right]-\left[\frac12\right]-\left[H-\frac1q\right]\\
&=H-0-(H-1)=1.
\end{aligned}
\]
Here $q\ge2$ and $H>0$. This contradicts \eqref{eq:3-16}. Hence $k>\ell_r-q$, proving \eqref{eq:3-20}.

Finally, compare $k$ and $k'$. By \eqref{eq:3-20}, both $\ell_r-k$ and $\ell_{r'}-k'$ lie between $0$ and $q-1$, so
\[
\bigl|(\ell_r-k)-(\ell_{r'}-k')\bigr|\le q-1.
\]
Combining this with \eqref{eq:3-19}, we obtain
\[
\begin{aligned}
|k-k'|
&=\bigl|(\ell_r-\ell_{r'})-\bigl((\ell_r-k)-(\ell_{r'}-k')\bigr)\bigr|\\
&\le|\ell_r-\ell_{r'}|+\bigl|(\ell_r-k)-(\ell_{r'}-k')\bigr|\\
&\le2|r-r'|+(q-1)\\
&\le2q+(q-1)=3q-1<3q.
\end{aligned}
\]
\end{proof}

\begin{proof}[Proof of Theorem~\ref{thm:3-1}]
If there is no edge of slope $1/q$, the conclusion holds. Otherwise, denote the degrees of its endpoints by $k_-<k_+$. By Lemma~\ref{lem:3-2}, there exist $r_-,r_+$ such that both $(r_-,k_-)$ and $(r_+,k_+)$ satisfy \eqref{eq:3-5}. Lemma~\ref{lem:3-5} then yields
\[
L_{I_{m,n}}(1/q)=k_+-k_-\le3q-1<3q.
\]
Moreover, since $a_{k_-},a_{k_+}\in\mathbb Q^\times$,
\[
\frac{L_{I_{m,n}}(1/q)}q
=\nu(a_{k_+})-\nu(a_{k_-})\in\mathbb Z,\qquad
0<\frac{L_{I_{m,n}}(1/q)}q<3.
\]
Therefore $L_{I_{m,n}}(1/q)\le2q$.
\end{proof}

\section{Proof of the common root conjecture}

In this section, we prove that the existence of a common nonzero root would force the edge length in Theorem~\ref{thm:3-1} to be at least $3q$.

\begin{lemma}\label{lem:4-1}
Let $1\le m<n$. If $P_m$ and $P_n$ have a common nonzero complex root, then there exist $q=2^e\ge2$ and a monic irreducible polynomial $g\in\mathbb{Q}_2[X]$ such that the $e$th binary digits of both $m$ and $n$ are $1$, and
\begin{equation}\label{eq:4-1}
g\mid Q_m,\qquad g\mid Q_n,
\end{equation}
\begin{equation}\label{eq:4-2}
\nu(\alpha)=-\frac1q\quad(g(\alpha)=0),
\qquad \deg g\ge q.
\end{equation}
\end{lemma}

\begin{proof}
Let $z\ne0$ be a common root, and let $F$ be the monic minimal polynomial of $(z-1)/2$ over $\mathbb{Q}$. Then
\begin{equation}\label{eq:4-3}
F\mid Q_m,\qquad F\mid Q_n,\qquad F\ne X+\frac12.
\end{equation}
Here $F\ne X+1/2$ because $(z-1)/2\ne-1/2$. Choose a monic irreducible factor $g$ of $F$ in $\mathbb{Q}_2[X]$. This gives~\eqref{eq:4-1}.

By uniqueness of the extension of the valuation, the conjugate roots of $g$ have the same valuation. Applying the root valuation theorem, Theorem~\ref{thm:2-2}, and Wahab's theorem, Theorem~\ref{thm:2-4}, to $g\mid Q_m,Q_n$, we obtain
\[
\nu(\alpha)=-1/q\quad(g(\alpha)=0),\qquad q=2^e,
\]
where the $e$th binary digits of both $m$ and $n$ are $1$.

We now show that $q\ne1$. Otherwise, $m$ and $n$ are both odd. By parity and Theorem~\ref{thm:2-4},
\[
Q_m(-1/2)=P_m(0)=0,\qquad
\nu(-1/2)=-1,\qquad L_{Q_m}(1)=1.
\]
By Theorem~\ref{thm:2-2}, $Q_m$ has exactly one root of valuation $-1$, counted with multiplicity, namely $-1/2$. Thus, using $g\mid F$ and the fact that $F$ is monic and irreducible, we obtain
\[
g=X+\frac12
\quad\Longrightarrow\quad F(-1/2)=0
\quad\Longrightarrow\quad F=X+\frac12,
\]
contrary to~\eqref{eq:4-3}. Therefore $q\ge2$.

Finally, we estimate $\deg g$. Put $d=\deg g$, and denote the roots of $g$ by $\alpha_1,\ldots,\alpha_d$. Since $g\mid Q_m$ and $Q_m(0)=1$, we have $g(0)\in\mathbb{Q}_2^{\times}$. Vieta's formula gives
\[
g(0)=(-1)^d\prod_{i=1}^d\alpha_i,
\]
and hence
\begin{equation}\label{eq:4-4}
\nu(g(0))=\sum_{i=1}^d\nu(\alpha_i)=-\frac dq\in\mathbb{Z}.
\end{equation}
Thus $q\mid d$, and $d\ge1$, so $d\ge q$.
\end{proof}

\begin{lemma}\label{lem:4-2}
Let $0\le m<n$, and let $g\in\mathbb{Q}_2[X]$ be a monic irreducible polynomial dividing both $Q_m$ and $Q_n$. Then
\begin{equation}\label{eq:4-5}
g^3\mid I_{m,n}.
\end{equation}
\end{lemma}

\begin{proof}
We first derive an expression for $I_{m,n}$ from the Legendre differential equation, which is~\cite[Eq.~(14.2.1)]{DLMF}
\[
(1-x^2)P_j''-2xP_j'+j(j+1)P_j=0.
\]
Substituting $x=2X+1$ and using
\[
P_j'(2X+1)=\frac12Q_j'(X),\qquad
P_j''(2X+1)=\frac14Q_j''(X),
\]
we obtain
\begin{equation}\label{eq:4-6}
X(X+1)Q_j''+(2X+1)Q_j'=j(j+1)Q_j.
\end{equation}
Put
\[
\Delta=n(n+1)-m(m+1)>0,\qquad
W=Q_m'Q_n-Q_mQ_n'.
\]
By the definition of $W$,
\[
\begin{aligned}
W'&=Q_m''Q_n+Q_m'Q_n'-Q_m'Q_n'-Q_mQ_n''\\
&=Q_m''Q_n-Q_mQ_n''.
\end{aligned}
\]
Multiplying the two instances of~\eqref{eq:4-6} with $j=m,n$ by $Q_n,Q_m$, respectively, and subtracting, we obtain
\[
\begin{aligned}
X(X+1)W'+(2X+1)W
&=\bigl(m(m+1)-n(n+1)\bigr)Q_mQ_n\\
&=-\Delta Q_mQ_n.
\end{aligned}
\]
Consequently,
\begin{equation}\label{eq:4-7}
\bigl(X(X+1)W\bigr)'=-\Delta Q_mQ_n.
\end{equation}
Since $X(X+1)W(X)$ vanishes at $X=0$, integrating from $0$ to $X$ and then using~\eqref{eq:1-4} gives
\begin{equation}\label{eq:4-8}
XI_{m,n}(X)=\int_0^X Q_m(t)Q_n(t)\,dt
=-\frac{X(X+1)W(X)}{\Delta}.
\end{equation}

We next prove that every root of $g$ has multiplicity at least three in $I_{m,n}$. Let $\alpha$ be any root of $g$. Since $g\mid Q_m,Q_n$,
\[
Q_m(\alpha)=Q_n(\alpha)=W(\alpha)=0.
\]
Substituting these equalities into~\eqref{eq:4-8} and its derivatives, we obtain
\[
\begin{aligned}
(XI_{m,n})(\alpha)
&=-\frac{\alpha(\alpha+1)W(\alpha)}{\Delta}=0,\\
(XI_{m,n})'(\alpha)
&=Q_m(\alpha)Q_n(\alpha)=0,\\
(XI_{m,n})''(\alpha)
&=Q_m'(\alpha)Q_n(\alpha)+Q_m(\alpha)Q_n'(\alpha)=0.
\end{aligned}
\]
Thus $\alpha$ is a root of $XI_{m,n}$ of multiplicity at least three. Moreover,
\[
Q_m(0)=1\quad\Longrightarrow\quad\alpha\ne0,
\]
so $\alpha$ is also a root of $I_{m,n}$ of multiplicity at least three. Since $g$ is separable in characteristic zero,
\[
g(X)^3=\prod_{g(\alpha)=0}(X-\alpha)^3\mid I_{m,n}(X).
\]
Dividing $I_{m,n}$ by the monic polynomial $g^3$ in $\mathbb{Q}_2[X]$, the remainder is zero in the algebraic closure and is therefore itself zero. This proves~\eqref{eq:4-5}.
\end{proof}

\begin{proof}[Proof of Theorem~\ref{thm:1-1}]
If $m=0$, then $P_0=1$, and the assertion holds. Let $1\le m<n$, and suppose that there is a common nonzero root. By Lemma~\ref{lem:4-1}, there exist $q=2^e\ge2$ and a corresponding factor $g$ such that the $e$th binary digits of both $m$ and $n$ are $1$, and
\[
d=\deg g\ge q,\qquad
\nu(\alpha)=-1/q\quad(g(\alpha)=0).
\]
By Theorem~\ref{thm:2-2} and Lemma~\ref{lem:4-2},
\[
L_g(1/q)=d,\qquad
I_{m,n}=g^3J\quad\text{for some }J\in\mathbb{Q}_2[X].
\]
All factors have nonzero constant terms. Hence the multiplication theorem, Theorem~\ref{thm:2-3}, gives
\[
L_{I_{m,n}}(1/q)=3L_g(1/q)+L_J(1/q)\ge3d\ge3q.
\]
On the other hand, Theorem~\ref{thm:3-1} gives
\[
L_{I_{m,n}}(1/q)<3q.
\]
This is a contradiction, so $P_m$ and $P_n$ have no common nonzero root.

It remains to determine what happens at zero. Taking the constant and linear coefficients in Rodrigues' formula~\eqref{eq:1-1}, we obtain
\[
\begin{aligned}
P_{2j}(0)&=\frac{(-1)^j}{4^j}\binom{2j}{j}\ne0,\\
P_{2j+1}(0)&=0,\\
P_{2j+1}'(0)&=\frac{(-1)^j(2j+1)}{4^j}\binom{2j}{j}\ne0.
\end{aligned}
\]
Therefore, $0$ is a simple root of every Legendre polynomial of odd degree and is not a root of any Legendre polynomial of even degree. Together with the absence of common nonzero roots, this gives the monic greatest common divisor stated in the theorem.
\end{proof}

\end{document}